\documentclass[12pt, a4paper]{article}
\usepackage[T1]{fontenc}
\usepackage{latexsym}
\usepackage{amssymb}
\usepackage{amsmath}
\usepackage{amsthm}
\usepackage{mathrsfs}
\usepackage{ucs}

\newtheorem{sat}{{\sc Theorem}}[section]
\newtheorem{lem}[sat]{{\sc Lemma} }
\newtheorem{kor}[sat]{Corollary} 
\newtheorem{defi}[sat]{Definition}   
\newcounter{saveeqn}
\newcommand{\arabeqn}{\setcounter{saveeqn}{\value{equation}}
\setcounter{equation}{0}
\renewcommand{\theequation}{\mbox{\arabic{section}.
\arabic{equation}}}}
\title{Density of spaces of trigonometric polynomials with frequencies from a subgroup in $L^\alpha$-spaces
}
\date{}
\begin{document}
\author{{\bf {\footnotesize BY}}\\ {\bf {\normalsize J. M. Medina}}~{\sc (Buenos Aires)},
{\bf {\normalsize L. Klotz}}~{\sc (Leipzig)}\\
{\footnotesize AND} {\bf {\normalsize M. Riedel}}~{\sc (Leipzig)}}
\maketitle
\noindent
{\it Abstract.} {\small 
Let $G$ be an LCA group, $H$ a closed subgroup, $\varGamma$ the dual group of $G$ and $\mu$ be a regular finite non-negative Borel measure
on  $\varGamma$. We give some necessary and sufficient conditions for the density of the set of trigonometric polynomials
on $\varGamma$ with frequencies from $H$ in the space  $L^\alpha(\mu), \alpha \in (0,\infty)$. 
}\\[0.2cm]
\setcounter{section}{0}
\section{Introduction}
\arabeqn
\noindent
Let $G$ be an LCA group, i.~e. a locally compact abelian group with Hausdorff topology, whose group operation is written additively.
Denote  by $\varGamma$  the dual group of $G$ and by $\langle\gamma,x\rangle$ the value of $\gamma\in \varGamma$ at $x \in G$. 
Let $H$ be a closed subgroup of $G$ and $ \varLambda:=\{\gamma \in \varGamma: \langle\gamma,y \rangle=1$ for all $y\in H \}$ its annihilator.   Recall that $\varLambda$ is a closed subgroup 
of $\varGamma$ and 
that the factor group $\varGamma/ \varLambda$ is (algebraically and topologically) isomorphic to the dual group of $H$, cf. \cite{Hewitt63}.

A trigonometric $H$-polynomial is a function $p:\varGamma\to \mathbb{C}$, which is a finite sum of the form 
$p(\cdot)=\sum a_k\langle\cdot,y_k\rangle$, where $a_k\in  \mathbb{C}$, $y_k \in H$. Denote by $\mathbf{P}(H)$ the linear space of all trigonometric 
$H$-polynomials.

If $\mu$ is a regular finite  Borel measure on $\varGamma$ and $\alpha \in (0,\infty)$, let $L^\alpha(\mu)$ be the metric space
of ($\mu$-equivalence classes of) Borel measurable $\mathbb{C}$-valued functions on $\varGamma$, which are $\alpha$-integrable with respect to $\mu$. For $\alpha \in (1,2]$, 
the space $L^\alpha(\mu)$ can be interpreted  as the spectral domain of a harmonizable symmetric $\alpha$-stable process, particularly, 
$L^2(\mu)$ is the spectral domain of a certain stationary Gaussian process. From  the prediction theory of such processes 
it arises the problem to describe those measures, for which  $\mathbf{P}(H)$ is dense in  $L^\alpha(\mu)$. The paper \cite{Klotz05}  
is an  special  case when $G=\mathbb{Z}$ and  prompts a result of the following form.

We say that a measure $\mu$ is concentrated on a transversal if there exists  a Borel subset $D$ of $\varGamma$ such that $\mu(\varGamma\setminus D)=0$ and 
$D\cap (\lambda+D)= \emptyset$ for all $\lambda \in \varLambda\setminus \{0\}$. Then  $\mathbf{P}(H)$ is dense in  $L^\alpha(\mu)$ for all 
$\alpha \in (0,\infty)$ if and only if  $\mu$ is concentrated on a transversal. The main goal of the present paper is to show that under several additional 
assumptions on $\varGamma$  or $\varLambda$ or $\mu$ the condition that $\mu$ is concentrated on a transversal is, indeed, a necessary or sufficient 
condition for  the density of $\mathbf{P}(H)$.

Moreover,  this description is equivalent to the problem of sampling and reconstructing a harmonizable symmetric $\alpha$-stable process. The first result in this direction was given by Lloyd \cite{Lloyd59} 
for wide sense stationary processes on $G=\mathbb{R}$ generalizing the  Whittaker-Kotel'nikov-Shannon (WKS) 
sampling theorem for $L^2(\mathbb{R})$ band limited functions. If $G$ is an LCA group, the WKS theorem was proved by Kluv\'{a}nek \cite{Kluva65} for $L^2 (G)$ and by Lee \cite{Lee78} for rather general processes of finite variance. In the context of  Hilbert spaces, other advances about sampling  in $L^2(G)$ can be found in e.g.
\cite{Fuhr2007} and for continuous groups of operators in e.g. \cite{Fer2015} .   
   
Section 2 contains definitions and basic facts from measure theory on topological spaces. Particularly, we prove  
some results on regular measures. Although they are simple, 
it seems that they are not often stated explicitly in the literature. 
We also recall some properties of transition probabilities, which are applied in Section 4.

The main result of Section 3 is the following. The space $\mathbf{P}(H)$ is dense in $L^\alpha(\mu)$ for all regular 
finite non-negative Borel measures $\mu$, which are concentrated on a transversal, and all $\alpha \in (0,\infty)$  if the annihilator group $\varLambda$ is metrizable. 
We mention that $\varLambda$ is metrizable if and only if the factor group $G/H$ is $\sigma$-compact, 
cf. \cite[Theorem~29]{Morris77},
and  that according to a fundamental structure theorem of LCA groups any compactly generated group is  
$\sigma$-compact, cf. \cite[(9.8)]{Hewitt63}.

Section 4 is devoted to the assertions that in case of a Polish space $\varGamma$ or a countable set $\varLambda$
the measure  $\mu$ is concentrated on a transversal if $\mathbf{P}(H)$ is dense in  $L^\alpha(\mu)$ for some $\alpha \in (0,\infty)$. If
$\varGamma$ is a Polish space, the proof heavily leans on properties of transition probabilities. If $\varLambda$ is countable, 
the corresponding proof for the circle group, cf. \cite[Theorem~3.4]{Klotz05}, can be adapted straightforwardly. 
\section{Some preliminaries from Borel measures on topological spaces}
\arabeqn
\noindent
We recall some definitions and facts from measure theory on topological spaces 
since a few of the notions are used in different ways in the literature.

Let $X$  be a topological Hausdorff space and $\mathbf{B}(X)$ the Borel $\sigma$-algebra of $X$, i.~e. $\mathbf{B}(X)$ 
is the $\sigma$-algebra
generated  by the open subsets of $X$. If $S$ is a subspace of $X$, then $\mathbf{B}(X)\cap S=\mathbf{B}(S)$, cf. \cite[13.5]{Part77}. 
Particularly, if $ S\in \mathbf{B}(X)$, then $\mathbf{B}(X)\cap S:=\{B\in \mathbf{B}(X): B\subseteq S\}$. The symbol $1_{S}$ stands for the indicator function of the set $S$.

A finite non-negative Borel measure on $\mathbf{B}(X)$ is called regular if for all  $B \in \mathbf{B}(X)$ and all $\varepsilon >0$,
there exist a compact set $C$ and an open set $U$ such that $C\subseteq B\subseteq U$ and $\mu(U\setminus C) <\varepsilon$, 
cf. \cite [p. 206]{Cohn80}. 
It is called  discrete if $\mu(X\setminus S)=0$ for some countable subset $S$.  Obviously, any discrete measure is regular. 
A $\mathbb{C}$-valued measure $\nu$ on $\mathbf{B}(X)$ is called regular if its variation, which  is denoted by $|\nu|$, is regular.
If $Y$ is a topological Hausdorff space and $\pi: X\to Y $ is a $(\mathbf{B}(X),\mathbf{B}(Y) )$-measurable map from $X$ to $Y$, we 
denote by $\nu\pi^{-1}$ the image measure of $\nu$ under $\pi$.

\begin{lem} \label{l2.1}
Let $X$ and $Y$ be topological Hausdorff spaces, $\mu$ and $\nu$ be a finite non-negative and a $\mathbb{C}$-valued, respectively, measure
on $\mathbf{B}(X)$ and  $\pi: X\to Y $  a continuous map. The following assertions are true.
\renewcommand{\labelenumi}{(\roman{enumi})}
\begin{enumerate}
\item If $S\in  \mathbf{B}(X)$ and $\mu$ is regular, then the restriction of  $\mu$ to $\mathbf{B}(S)$ is regular.
\item If   $\mu$ is regular, then $\mu\pi^{-1}$ is regular.
\item If  $\nu$ is absolutely continuous with respect to $\mu$ and $\mu$ is regular, then $\nu$ is regular. 
\item If   $\nu$ is regular, then $\nu\pi^{-1}$ is regular.
\item If $(\mu_k)_{k\in \mathbb{N}}$ is a sequence of regular  non-negative measures on $ \mathbf{B}(X)$ such that 
$\mu=\sum_{k=1}^{\infty}\mu_k$, then $\mu$ is regular.
\end{enumerate}
 \end{lem}
\begin{proof}
(i) Let $ B \in  \mathbf{B}(S)$. Since $\mu$ is regular, for $\varepsilon >0$ there exist a compact set $C$ 
and an open set $U$ such that $C\subseteq B\subseteq U$ and $\mu(U\setminus C) <\varepsilon$. Since $U_S:=U\cap S$ is 
an open subset of $S$ and $\mu(U_S\setminus C)\le \mu(U\setminus C)$, the assertion follows.\\
(ii) Let $ \tilde{B} \in  \mathbf{B}(Y)$ and $B:=\pi^{-1}( \tilde{B} )$. Let $\varepsilon >0$ and $C\subseteq B$ be compact with 
$\mu(B\setminus C) < \frac{\varepsilon}{2}$. Then $\mu\pi^{-1}( \tilde{B}\setminus \pi(C))\le \mu(B\setminus C) < \frac{\varepsilon}{2}$
for the compact subset  $\pi(C)$ of $Y$. It follows that for the set $\tilde{B}^{c}:=Y\setminus \tilde{B}$, there exists 
a compact  set $\tilde{C}$ satisfying $\tilde{C}\setminus \tilde{B}^{c} $ and
$\mu( \tilde{B}\setminus \tilde{C} ) <\frac{\varepsilon}{2}$ The set  $\tilde{C}^{c}$  is open  since $Y$ was assumed to be a Hausdorff
space. We obtain $\pi(C)\subseteq\tilde{B}\subseteq\tilde{C}^{c}$ and 
$\mu\pi^{-1}(\tilde{C}^{c}  \setminus  \pi(C))\le \mu\pi^{-1}(\tilde{C}^{c}  \setminus  \tilde{B})+
\mu\pi^{-1}( \tilde{B} \setminus\pi(C)) < \varepsilon$, which implies that  $\mu\pi^{-1}$ is regular.\\
(iii) The result is an immediate consequence of the fact that if $\nu$ is absolutely continuous with respect to $\mu$, then for 
 $\varepsilon >0$, there exists $\delta >0$ such that for all $  B \in  \mathbf{B}(X)$, from $\mu(B) <\delta$ it follows 
 $|\nu|(B) <\varepsilon$, cf. \cite[Lemma~4.2.1]{Cohn80}. \\
(iv) Apply (ii) and (iii).\\
(v) For $\varepsilon >0$ there exists $n\in \mathbb{N}$ such that $\sum_{k=n+1}^{\infty}\mu_{k}(X) < \frac{\varepsilon}{2}$. If 
$ B \in  \mathbf{B}(X)$, the regularity of $\mu_{k}$ yields the existence of a compact set $C_{k}$ and an open set $U_{k}$
satisfying $C_{k}\subseteq B\subseteq U_{k}$ and   
$\mu_{k}(U_{k}\setminus C_{k}) <\frac{\varepsilon}{2^{k+1}}$, $k\in  \mathbb{N}$.
The set $C:=\bigcup_{k=1}^{n}C_k$ is compact, the set $U:=\bigcap_{k=1}^{n}U_{k}$ is open, 
$C\subseteq B \subseteq U$, and $\mu(U\setminus C) < \varepsilon$.
\end{proof}
Let $X$ and $Y$  be Polish spaces,  i.~e. separable topological spaces that can be metrized by means of a complete 
metric. Note that any finite non-negative measure on $ \mathbf{B}(X)$ is regular, cf. \cite[Proposition~8.1.10]{Cohn80}. 
In what follows any regular finite non-negative Borel measure will be simply called a measure.

We need some facts on transition probabilities. A general theory of conditional probabilities and transition probabilities is 
given in \cite{Rao93}. Since we need only special cases here, we refer to the excellent introduction \cite{Part77}.
\begin{defi} (cf. \cite [35.4]{Part77}) A map $w: Y\times\mathbf{B}(X) \to [0,1)$ is called  a transition probability if 
\renewcommand{\labelenumi}{(\roman{enumi})}
\begin{enumerate}
\item for all $B\in \mathbf{B}(X)$, the function $y\to w(y,B)$, $y\in Y$, is $(\mathbf{B}(X), \mathbf{B}([0,1]))$-measurable,
\item for all $y\in Y$, the map $B\to w(y,B)$, $B\in \mathbf{B}(X)$, is a probability measure.
\end{enumerate}
\end{defi}
\begin{sat}\label{Th2.3}
(cf. \cite[35.9, 35.11, 35.14]{Part77} and \cite[Proposition~7.6.2]{Cohn80})~\\
Let $w$ be a transition probability on $Y \times \mathbf{B}(X)$ and $\tilde{\mu}$ be a measure on $\mathbf{B}(Y)$.
\renewcommand{\labelenumi}{(\roman{enumi})}
\begin{enumerate}
\item If  $f:X\times Y \to [0,\infty)$ is a $( \mathbf{B}(X \times Y), \mathbf{B}([0,\infty]))$-measurable function, 
then  for all $y\in Y$ the integral $\int_{X}f(x,y) w(y,dx)\le \infty$ exists and the function $y\to \int_{X}f(x,y) w(y,dx)$, $y\in Y$,
is  $( \mathbf{B}(Y), \mathbf{B}([0,\infty]))$-measurable.
\item Setting $\varrho(B):=\int_{Y}\int_{X}1_{B}(x,y) w(y,dx)\tilde{\mu}(dy), B\in \mathbf{B}(X \times Y),$ a measure $\varrho$  
is defined. A  $( \mathbf{B}(X \times Y),\mathbf{B}(\mathbb{C}))$-measurable function $g:X\times Y \to \mathbb{C}$ is integrable with 
respect to $\varrho$ if and only if $\int_{X}|g(x,y)| w(y,dx) < \infty$ for $\tilde{\mu}$-a.a. $y\in Y$ and 
$\int_{Y}\int_{X}|g(x,y)| w(y,dx)\tilde{\mu}(dy) <\infty$. In this case 
$$
\int_{X\times Y}g(x,y) \varrho (dx\otimes dy)= \int_{Y}\int_{X}g(x,y) w(y,dx)\tilde{\mu}(dy).
$$
\end{enumerate}
\end{sat}

Denote by $\delta_x$ the Dirac measure on $ \mathbf{B}(X)$ concentrated at $x\in X$. If $\psi:Y\to X$ is a  
$( \mathbf{B}(Y),\mathbf{B}(X))$-measurable function, it is  not hard to see that by  $w(y,B):=\delta_{\psi(y)}(B)$, $y\in Y$,
$B \in \mathbf{B}(X)$, a transition probability on $Y\times \mathbf{B}(X) $ is defined and from Theorem~\ref{Th2.3}~(ii) 
we can easily obtain the following result.

\begin{kor}\label{c2.4} 
If $\tilde{\mu}$ is a measure on $\mathbf{B}(Y)$, then  
$$
\mu(B):=\int_{Y}\int_{X}1_{B\times Y}(x,y) \delta_{\psi(y)}(dx)\tilde{\mu}(dy), B\in \mathbf{B}(X),
$$ 
defines a measure $\mu$ on $\mathbf{B}(X)$ satisfying
$$
\int_{X}f(x) \mu(dx)=\int_{Y}\int_{X}f(x) \delta_{\psi(y)}(dx)\tilde{\mu}(dy)
$$
for all $f\in L^{1}(\mu)$. Particularly, $\mu(B)=\tilde{\mu}\psi^{-1}(B)$, $B\in\mathbf{B}(X)$.
\end{kor}

\begin{sat}\label{Th2.5} (cf. \cite[46.3]{Part77}) Let $X$ and $Y$ be Polish spaces, $\mu$ a measure on $ \mathbf{B}(X)$, and 
$\pi:X\to Y$ a $( \mathbf{B}(X),\mathbf{B}(Y))$-measurable map. There exists a transition probability $w$ on $Y\times  \mathbf{B}(X)$
with the following properties:
\renewcommand{\labelenumi}{(\roman{enumi})}
\begin{enumerate}
\item There exists a set $\tilde{B}_{0}\in  \mathbf{B}(Y) $ such that $\mu\pi^{-1}(\tilde{B}_{0})=0$ and moreover 
$w(y, \pi^{-1}(\{y\})=1$ for all $y\in Y\setminus \tilde{B}_{0}$.
\item For all $f \in L^{1}(\mu)$,
$$
\int_{Y}\int_{X}f(x) w(y,dx)\mu\pi^{-1}(dy)=\int_{X}f(x) \mu(dx).
$$
\end{enumerate}
\end{sat}
\begin{kor}  \label{c2.6}
Let $\alpha\in (0,\infty)$. Assume that the conditions of Theorem~\ref{Th2.3} are satisfied and let $w$ be  
a  transition probability on $Y\times \mathbf{B}(X) $ with properties (i) and (ii)  of Theorem~\ref{Th2.5}. 
If a family of functions $\{f_j: j\in I\}$ is dense  in $L^\alpha(\mu)$, there exists a set  $\tilde{B}\in  \mathbf{B}(Y) $
such that  $\tilde{B}_{0}\subseteq  \tilde{B}$,  $\mu\pi^{-1}(\tilde{B})=0$, and $\{f_j: j\in I\}$ is dense in $L^\alpha(w(y,\cdot)$
for all $y\in Y\setminus \tilde{B}$.
\end{kor}
\begin{proof}
Since  $X$ is a separable space, its Borel $\sigma$-algebra is countably generated. Therefore, there exists  a countable set $S:=
\{s_k: k\in \mathbb{N}\}$ of $(\mathbf{B}(X),\mathbf{B}(\mathbb{C}))$-measurable step functions, which  is dense  in $L^\alpha(w(y,\cdot))$
for all $y\in Y\setminus \tilde{B}_{0}$, cf. the proof of Proposition~3.4.5 of \cite{Cohn80}. Since  $\{f_j: j\in I\}$ is dense  
in $L^\alpha(\mu)$, for any $k \in\mathbb{N}$, there exists a sequence $(j_{n,k})_{n\in \mathbb{N}}$ of indices such that 
$$
\lim_{n\to \infty}\int_{X}|s_k(x)-f_{j_{n,k}}(x)|^{\alpha}\mu(dx)=0.
$$
Define a function $g_{n,k }:Y \to [0,\infty]$ by
$
g_{n,k }(y):=\int_{X}|s_k(x)-f_{j_{n,k}}(x)|^{\alpha}w(y,dx)
$, $y\in Y$. According to Theorem~\ref{Th2.5}~(ii), for all $k\in \mathbb{N}$, the sequence $(g_{n,k})_{n\in \mathbb{N}}$
tends to zero in $L^{1}(\mu\pi^{-1})$. Proceeding to a suitable subsequence if necessary, we can assume that for $k\in \mathbb{N}$,
there exists $\tilde{B}_{k}\in \mathbf{B}(Y)$ such that $\tilde{B}_{k}\subseteq Y\setminus \tilde{B}_{0}$, $\mu\pi^{-1}(\tilde{B}_{k})=0$.
and $\lim_{n\to \infty} g_{n,k}(y)=0$ for all $y\in \tilde{B}_{k}$. Obviously,  the set 
$\tilde{B}=\bigcup_{k=0}^{\infty}\tilde{B}_{k}$ has all properties claimed.
\end{proof}

\section{Conditions under which $\mathbf{P}(H)$ is dense if $\mu$ is concentrated on a transversal}
\arabeqn
\noindent
Let $G$ be an LCA group, $H$ its closed subgroup, $\varGamma$ the dual group of $G$, $\varLambda$ the annihilator of $H$, $\pi$ the canonical 
homeomorphism from $\varGamma$ onto $\varGamma/\varLambda$, and denote $\pi(\gamma)=:\tilde{\gamma}$, $\gamma\in  \varGamma$.
\begin{lem}\label{l3.1}
Let $\nu$ be a $\mathbb{C}$-valued measure on $ \mathbf{B}(\varGamma)$. Assume that $|\nu|\pi^{-1}$ is discrete. If $|\nu|$ is concentrated on a transversal and 
$\int_{\varGamma}\langle\gamma,y\rangle\nu(d\gamma)=0$ for all $y\in H$, then $\nu$  is the zero measure.
\end{lem} 
\begin{proof}
The elements of the dual group of   $\varGamma/\varLambda$ are precisely all functions $\chi_{y}:\varGamma/\varLambda \to\mathbb{C}$,
$y\in H$, of the form $ \chi_{y}(\tilde{\gamma})=\langle\gamma,y\rangle$, $\tilde{\gamma}\in \varGamma/\varLambda$, where $\gamma$
can be chosen arbitrarily from $\pi^{-1}(\{\tilde{\gamma}\})$. By the integral transformation formula 
$$
\int_{\varGamma/\varLambda}\chi_{y}(\tilde{\gamma})\nu\pi^{-1}(d\tilde{\gamma})=\int_{\varGamma}\langle\gamma,y\rangle\nu(d\gamma),
$$
for all $y\in H $, and from the uniqueness theorem of the Fourier transform, cf. \cite[(31.5)]{Hewitt70}, it follows that $\nu\pi^{-1}$
 is the zero measure. It is not hard to see that if $| \nu|\pi^{-1}$ is discrete and $|\nu|$ is concentrated on a transversal, then there exists a finite  or countably infinite 
 subset $B:=\{\gamma_{k}: k\in K \}$ of $\varGamma$, which meets each $\varLambda$-coset at most  once and such that 
 $|\nu|(\varGamma\setminus B)=0$ and $|\nu|(\{\gamma_{k}\}) >0$, $k \in K$. Let $h$ be the Radon-Nikodym derivative of $\nu$ with respect 
 to $|\nu|$. We can assume  that $h=0$ on $ \varGamma\setminus B$ and define  a function $\tilde{f}:\varGamma/\varLambda\to \mathbb{C}$
 by setting
 $$
\tilde{f}(\tilde{\gamma})=\left\{
\begin{array}{  c c c}
h(\gamma_{k})^{-1},&\mbox{if}\:\gamma_{k}\in \pi^{-1}(\{\tilde{\gamma}\}),&k\in K,\\
0,&&\mbox{else}
\end{array}
\right..
 $$
Obviously,  $\tilde{f}$ is a $(\mathbf{B}(\varGamma/\varLambda), \mathbf{B}(\mathbb{C}))$-measurable function  and  
\begin{eqnarray*}
|\nu|(\varGamma)&=&|\nu|(B)=\int_{B}h(\gamma)^{-1} \nu(d\gamma)
\\&=&\int_{B}\tilde{f}(\tilde{\gamma})\nu(d\gamma)
=\int_{\varGamma}\tilde{f}(\pi{\gamma})\nu(d\gamma)
\\&=&\int_{\varGamma/\varLambda}\tilde{f}(\tilde{\gamma})\nu\pi^{-1}(d\tilde{\gamma})=0
\end{eqnarray*}
by the integral transformation formula.
\end{proof}
\begin{sat}\label{Th3.2}
Let $\mu$ be a measure on  $ \mathbf{B}(\varGamma)$ and $\alpha \in (0,\infty)$. Assume that $\mu\pi^{-1}$ is discrete. If  $\mu$  is concentrated on a transversal,
then $\mathbf{P}(H)$ is dense in $L^\alpha(\mu)$.
\end{sat}
\begin{proof}
First let   $\alpha \in [1,\infty)$ and $ \beta:= \frac{\alpha}{\alpha-1}\in (1,\infty]$. If  $f \in L^{\beta}(\mu)$ is such that
$$
\int_{\varGamma}\langle\gamma,y \rangle f(\gamma)\mu(d\gamma)=0,
$$
for all $y \in H$, then by Lemma~\ref{l3.1} the $\mathbb{C}$-valued measure $fd\mu$ is the zero measure, which implies that $\mathbf{P}(H)$ 
is dense in $L^\alpha(\mu)$. Since $\mu$ is finite, it follows easily that for $\alpha \in (0,1)$,  the space $\mathbf{P}(H)$ is dense
in $L^\alpha(\mu)$ as well.
\end{proof}
The preceding proof shows that if the assumption of discreteness of $\nu\pi^{-1}$ could be dropped, one could establish Theorem~\ref{Th3.2} 
for all measures $\mu$ on $ \mathbf{B}(\varGamma)$. In fact, a similar result to that of Lemma~\ref{l3.1} was formulated by 
Lee \cite[Lemma~2]{Lee78} without assuming that $|\nu|\pi^{-1}$ is discrete. However, his proof seems to contain 
a gap since he did not prove the Borel measurability of a function which  corresponds to the function $\tilde{f}$ of our proof of 
Lemma~\ref{l3.1}. The remaining part of the present section is devoted to the proof of the assertion that if $\varLambda$ is 
metrizable, then  $\mathbf{P}(H)$ is dense in  $L^\alpha(\mu)$ for all measures $\mu$, which are concentrated on a transversal, and all $\alpha \in (0,\infty)$. To obtain 
such  a result we apply a theorem of Feldman and Greenleaf on the existence of a Borel transversal and a Borel measurable cross-section, cf. \cite {Feld68}. 

A subset $T$ of $\varGamma$ is called a transversal (with respect to $\varLambda$) if it meets each $\varLambda$-coset just once.
A map $\tau: \varGamma/\varLambda \to T$ is called a cross-section if $\pi\circ \tau$ is the identical map on 
$ \varGamma/\varLambda$. We mention that Kluv\'{a}nek's sampling theorem and related results, cf. \cite{Bea07,Kluva65}, clarify the relationship between transversals and sampling theorems.
\begin{lem}\label{l3.3}
 A subset $T$ of $\varGamma$ is  a transversal if and only if 
 $$
 \pi ^{-1}(\pi(T))=\bigcup_{\lambda\in \varLambda}(\lambda+T)=\varGamma
 $$
 and $T\cap (\lambda+T)=\emptyset$ for all $\lambda\in \varLambda\setminus \{0\}$.
\end{lem}
\begin{proof}
Let $T$ be a transversal. Then for $\gamma\in \varGamma$, there exists $t \in T$ such that $t\in\gamma +\varLambda$, hence, 
$\gamma \in \lambda+T$ for some $\lambda \in \varLambda$, which yields $\varGamma=\pi ^{-1}(\pi(T)) $. If $\gamma \in T\cap (\lambda+T)$
for some $\gamma\in \varGamma$ and  $\lambda \in \varLambda$, then $\gamma=\lambda+t$  for some $t \in T$, hence, $\gamma \in t+\varLambda$.
Since $t\in t+\varLambda$, it follows $\gamma=t$ and $\lambda=0$. Let $T$ be a subset of $\varGamma$, which is not a transversal.  
Then there exists $\gamma\in \varGamma$ such that $(\gamma+\varLambda)\cap T=\emptyset$ or there exist $t_1, t_2 \in T$, $t_1\neq t_2$, with 
$t_1, t_2 \in \gamma+\varLambda$. In the first case 
$$
(\gamma+\varLambda)\cap(\lambda+T)=(\gamma+\lambda +\varLambda)\cap(\lambda+T)=\lambda+[(\gamma+\varLambda)\cap T]= \emptyset
$$ 
for all $\lambda \in \varLambda$, which yields $\pi ^{-1}(\pi(T))\neq \varGamma$.
In the second case there exist $\lambda_{k}\in \varLambda$ such that $t_{k}=\gamma+\lambda_{k}$, $k\in \{1,2\}$, hence, 
$T \cap (\lambda_{1}-\lambda_{2}+T)\neq \emptyset$, where $\lambda_{1}-\lambda_{2}\neq 0$ since $t_{1}\neq t_{2}$.
\end{proof}
\begin{lem}\label{l3.4}
If $R$ is a transversal and $S$ is a subset of $\varGamma$ such that $\quad$ $S\cap (\lambda+S)=\emptyset$ for all  
$\lambda\in \varLambda\setminus \{0\}$, then the set $T:=S\cup [R\cap(\varGamma\setminus\pi ^{-1}(\pi(S)) )] $
is a transversal.
\end{lem}
\begin{proof}
If $\gamma \in \lambda+S$ for some $ \lambda\in \varLambda$, then, of course, $\gamma\in \lambda+T$. If $\gamma \notin \lambda+S$ 
for all $\lambda\in \varLambda$, then
$$
\gamma \in \varGamma \setminus\pi ^{-1}(\pi(S))=\lambda+[\varGamma \setminus\pi ^{-1}(\pi(S))]
$$
for all  $\lambda \in \varLambda$. Since $R$ is a transversal, $\gamma\in \lambda+R$ for some $ \lambda\in \varLambda$, hence 
$\gamma\in \lambda+T$, which implies $\varGamma=\pi ^{-1}(\pi(T))$. If we have $\lambda+t_1=t_2$ for some $ \lambda\in \varLambda$,
$t_{1}, t_{2}\in T$, then either $t_{1}, t_{2}\in S$ or $t_{1}, t_{2}\in [R\cap(\varGamma\setminus\pi ^{-1}(\pi(S)) )]$. In both cases $\lambda=0$ by properties of $S$ or the transversal $R$, respectively. 
Thus, $T$ is a transversal  by Lemma~\ref{l3.3}.
\end{proof}
\begin{sat}(cf. \cite[Theorem~1 and Remark~3~(ii)]{Feld68} )\label{Th3.5} 
Let $\varLambda$ be a metrizable subgroup of $\varGamma$. There exists a transversal, which belongs to $ \mathbf{B}(\varGamma)$. If $T$ 
is such a transversal, then the corresponding cross-section $\tau$ can be chosen 
$( \mathbf{B}(\varGamma/\varLambda), \mathbf{B}(T))$-measurable.
\end{sat}
\begin{lem}\label{l3.6}
Let $\mu$ be a measure on $ \mathbf{B}(\varGamma)$ and $B\in\mathbf{B}(\varGamma)$. There exists a set $A\in \mathbf{B}(\varGamma) $
such that $A\subseteq B$, $\mu(B \setminus A )=0$, and $\pi ^{-1}(\pi(A))\in \mathbf{B}(\varGamma)$.
\end{lem}
\begin{proof}
By regularity of $\mu$ there exists a sequence $(C_{k})_{k\in \mathbb{N}}$ of compact subsets of $B$ such that 
$\lim_{k\to \infty}\mu(C_{k})=\mu(B)$. The set $A:= \bigcup_{k=1}^{\infty}C_{k}$ belongs to $ \mathbf{B}(\varGamma)$,
$A\subseteq B$, and  $\mu(B \setminus A )=0$. Since $\pi$ is continuous, $\pi(C_{k})$ is compact, hence, 
$\pi ^{-1}(\pi(A))\in \mathbf{B}(\varGamma)$.
\end{proof}

\begin{lem}\label{l3.7}
Let $\nu$ be a $\mathbb{C}$-valued measure on $\mathbf{B}(\varGamma)$. 
Assume that $\varLambda$ is metrizable. 
If $|\nu|$ is concentrated on a transversal and $\int_{\varGamma}\langle\gamma,y\rangle\nu(d\gamma)=0$
for all $y \in H$, then $\nu$ is the zero measure.
\end{lem}
\begin{proof}
From the proof of Lemma~\ref{l3.1} we know that $\nu\pi^{-1}$ is a zero measure. Let $R$ be a transversal, which belongs to
$\mathbf{B}(\varGamma)$, cf. Theorem~\ref{Th3.5}. Since $|\nu|$ is concentrated on a transversal, there exists $D\in \mathbf{B}(\varGamma)$ such that
$\mu(\varGamma \setminus D )=0$ and $D\cap(\lambda+D)=\emptyset$ for all $\lambda\in \varLambda\setminus\{0\}$. 
By Lemma~\ref{l3.4} the set $T:=D\cup [R\cap(\varGamma\setminus\pi ^{-1}(\pi(D)) ] $ is a transversal and by Lemma~\ref{l3.6} 
we can assume that $T\in  \mathbf{B}(\varGamma)$.
Let $\tau$ be a corresponding $( \mathbf{B}(\varGamma/\varLambda), \mathbf{B}(T))$-measurable cross-section. Let $h$ be the 
Radon-Nikodym derivative of $\nu$ with respect to  $|\nu|$. We can assume that $h$ is a
$( \mathbf{B}(\varGamma), \mathbf{B}(\mathbb{C}))$-measurable function, $|h|=1$ on $D$ and $h=0$ on $\varGamma\setminus D$. Therefore, 
the function $h^{+}$ defined by
$$
h^{+}(\gamma)=\left\{
\begin{array}{cc}
h(\gamma)^{-1},&\gamma\in D,\\
0,&\gamma\in\varGamma\setminus D,
\end{array}
\right.
$$ 
is $( \mathbf{B}(\varGamma), \mathbf{B}(\mathbb{C}))$-measurable, $\tilde{f}:=h^{+}\circ\tau$ is 
$( \mathbf{B}(\varGamma/\varLambda), \mathbf{B}(\mathbb{C}))$-measurable and we can complete the proof similarly 
to the proof of Lemma~\ref{l3.1}.
\end{proof}
As a consequence of the preceding lemma we obtain the following assertion, cf. the proof of Theorem~\ref{Th3.2}.
\begin{sat}\label{Th3.8}
Let $\mu$ be a measure on $ \mathbf{B}(\varGamma)$ and $\alpha \in (0,\infty)$.
Assume that the annihilator group $\varLambda$ is metrizable. If $\mu$ is concentrated on a transversal, then $\mathbf{P}(H)$ is dense in $L^\alpha(\mu)$.
\end{sat}
\section{Conditions under which $\mu$ is concentrated on a transversal if $\mathbf{P}(H)$ is dense}
To motivate our approach we again start with the case that $\mu\pi^{-1}$ is discrete.
\begin{sat}\label{Th4.1} 
Let $\mu$ be a measure on $ \mathbf{B}(\varGamma)$ and $\alpha \in (0,\infty)$. Assume that  $\mu\pi^{-1}$ is a discrete measure. 
Then   $\mathbf{P}(H)$ is dense in $L^\alpha(\mu)$ if and only if  $\mu$ is concentrated on a transversal. 
\end{sat}
\begin{proof}
 The ``if-part'' was stated in Theorem~\ref{Th3.2}. To prove the converse let 
 $\tilde{S}:=\{\tilde{\gamma}_{k}: k\in K\}$  be a finite or countably infinite subset of $\varGamma/\varLambda$ such that
 $\mu\pi^{-1}((\varGamma/\varLambda)\setminus\tilde{S})=0$ and $\mu\pi^{-1}(\tilde{\gamma}_{k}) >0$, $k\in K$.
 Let $\mu_{k}$ be the restriction of $\mu$ to  $ \mathbf{B}(\pi^{-1}(\{ \tilde{\gamma}_{k}\}))$, $p_k$ 
 be the restriction of $p\in \mathbf{P}(H)$ to  $\pi^{-1}(\{\tilde{\gamma}_{k}\})=:B_{k}$, and let   $\mathbf{P}_{k}(H)$
be the linear space of all such restrictions, $k\in K$. Then 
$$
\int_{\varGamma}|f(\gamma)|^{\alpha} \mu(d\gamma)=\sum_{k\in K}\int_{B_{k}}|f(\gamma)|^{\alpha} \mu_{k}(d\gamma)
$$
for all $f \in L^\alpha(\mu)$. It follows that if  $\mathbf{P}(H)$ is dense in $L^\alpha(\mu)$, then for all $k\in K$
the space  $\mathbf{P}_{k}(H)$ is dense in $L^\alpha(\mu_{k})$. Since the functions of $\mathbf{P}_{k}(H)$ are constants, 
the measure  $\mu_{k}$ has the form  $\mu_{k}=a_{k}\delta_{\gamma_{k}}$ for some  $a_{k} \in (0,\infty)$, 
$\gamma_{k}\in \pi^{-1}(\{\tilde{\gamma}_{k}\})$, $k\in K$. The set $D:=\{\gamma_{k}: k\in K\}$ belongs to $ \mathbf{B}(\varGamma)$,
$\mu(\varGamma\setminus D)=0$, and $D\cap(\lambda+D)= \emptyset$ for all $\lambda\in \varLambda\setminus\{0\}$.
\end{proof}
The preceding proof is based on the fact that if $\mu\pi^{-1}$ is  discrete, then the space  $L^\alpha(\mu)$ can be written 
as a direct sum of certain $L^\alpha$-spaces over $\varLambda$-cosets. Assuming that $\varGamma$ is a Polish space 
and applying the theory of transition probabilities, this idea can be generalized as follows.
\begin{lem}\label{l4.2} 
If $\varGamma$ is a   Polish space, then $\varGamma/\varLambda$ is a Polish space.
\end{lem}
\begin{proof}
A Polish space $\varGamma$ can be metrized by means of an invariant metric $\sigma$, cf.  \cite[(8.3)]{Hewitt63}. 
Therefore, the topology of $\varGamma/\varLambda$ can be metrized by a metric $\tilde{\sigma}$ defined by 
$$
\tilde{\sigma}(\tilde{\gamma},\tilde{\beta}):=\inf\{ \sigma(\gamma,\beta): \gamma\in \pi^{-1}\{\tilde{\gamma}\} ,
\beta\in \pi^{-1}\{\tilde{\beta}\}
\},
$$
cf. \cite[(8.14)(a), (b)]{Hewitt63}. Let $(\tilde{\gamma}_{n})_{n\in\mathbb{N} }$ be a Cauchy sequence with respect 
to the metric   $\tilde{\sigma}$. Since $\sigma$ is an invariant metric, one can construct  a Cauchy sequence
 $(\gamma_{n})_{n\in\mathbb{N} }$  such that $\gamma_{n} \in \pi^{-1}\{\tilde{\gamma}_n\}$ and 
 $$
 \sigma(\gamma_n,\gamma_{n+1}) < \tilde{\sigma}(\tilde{\gamma}_{n},\tilde{\gamma}_{n+1})+\frac{1}{2^{n}},\; n\in\mathbb{N}.
$$
Since the topology of $\varGamma$
can be metrized  by means of a complete metric, a theorem of Klee \cite[(2.4)]{Klee52} asserts that $\sigma$ is complete. It follows 
$$
\lim_{n\to \infty}\sigma(\gamma_0,\gamma_{n})=0
$$
for some $\gamma_{0}\in \varGamma$, hence,
$$
\lim_{n\to \infty}\tilde{\sigma}(\tilde{\gamma}_{0},\tilde{\gamma_{n}})=0,
$$
which means that $\tilde{\sigma}$ is complete.
Clearly, the image of a dense subset of $\varGamma$ under the map $\pi$
is dense  in $\varGamma/\varLambda$. Thus, if $\varGamma$ is separable, then $\varGamma/\varLambda$ is separable.
\end{proof}
\begin{lem}\label{l4.3}
Let $\varGamma$ be a   Polish space. A measure $\mu$ on  $ \mathbf{B}(\varGamma)$ is concentrated on a transversal if and only if there exists a  
$( \mathbf{B}(\varGamma/\varLambda)$, $ \mathbf{B}(\varGamma))$-measurable function $\tau:\varGamma/\varLambda \to \varGamma$
satisfying $\tau(\tilde{\gamma})\in \pi^{-1}(\{\tilde{\gamma} \})$, $\tilde{\gamma}\in \varGamma/\varLambda$, and 
$$
\int_{\varGamma}f(\gamma)\mu(d\gamma)=\int_{\varGamma/\varLambda}\int_{\varGamma} f(\gamma) \delta_{\tau(\tilde{\gamma})}(d\gamma) \mu\pi^{-1}
(d\tilde{\gamma})
$$
for all $f \in L^{1}(\mu)$.
\end{lem}
\begin{proof}
If there exists a function $\tau$ with the described properties, according to \cite[24.23]{Part77}
there exists a set $D\in  \mathbf{B}(\varGamma)$ with $D\subseteq \tau (\varGamma/\varLambda)$ and we have 
$\mu\pi^{-1}(\tau^{-1}(\varGamma\setminus D))=0$.
Since $\tau\left( \varGamma/\varLambda\right)$ contains just one element from each  $\varLambda$-coset, one  has 
$D\cap (\lambda+D)=\emptyset$ for $\lambda \in \varLambda\setminus\{0\}$. The equality 
$\mu(\varGamma\setminus D)=\mu\pi^{-1}(\tau^{-1}(\varGamma\setminus D))= 0$ is a consequence of Corollary~\ref{c2.4}.
Assume, conversely, that there exists $D\in  \mathbf{B}(\varGamma)$ such that $\mu(\varGamma\setminus D)=0$ and  
$D\cap (\lambda+D)=\emptyset$ for $\lambda \in \varLambda\setminus\{0\}$. Choose  a transversal $R\in \mathbf{B}(\varGamma)$
according to Theorem~\ref{Th3.5} and construct a transversal  $T:=D\cup [R\cap(\varGamma\setminus\pi ^{-1}(\pi(D))], T \in \mathbf{B}(\varGamma),$
by Lemmas~\ref{l3.4} and \ref{l3.6}. Let $\tau$ be a corresponding 
$(\mathbf{B}(\varGamma/\varLambda),\mathbf{B}(\varGamma) )$-measurable cross-section and define a measure $\varrho$ by 
$$
\varrho(B):=\int_{\varGamma/\varLambda}\int_{\varGamma}
1_{B\times \varGamma/\varLambda}(x,\tilde{\beta}) \delta_{\tau(\tilde{\beta} )}(d\gamma)\mu\pi^{-1}(d\tilde{\beta}), B\in\mathbf{B}(\varGamma),
$$
according to Corollary~\ref{c2.4}. Since
$$
\varrho(B)=\mu\pi^{-1}(\tau^{-1}(B))=\mu\pi^{-1}(\tau^{-1}(B\cap D))=\mu(B\cap D)=\mu(B), B\in\mathbf{B}(\varGamma),
$$
and since
$\tau(\tilde{\gamma})\in \pi^{-1}(\{ \tilde{\gamma}\})$, for all $\tilde{\gamma}\in \varGamma/\varLambda$, it follows
\begin{eqnarray*}
\int_{\varGamma}f(\gamma) \mu(d\gamma) &=&\int_{\varGamma}f(\gamma) \varrho(d\gamma)
=\int_{\varGamma/\varLambda}\int_{\varGamma}
f(\gamma) \delta_{\tau(\tilde{\beta} )}(d\gamma)\mu\pi^{-1}(d\tilde{\beta})
\\&=&\int_{\varGamma/\varLambda}\int_{\varGamma}
f(\gamma) \delta_{\tau(\tilde{\gamma} )}(d\gamma)\mu\pi^{-1}(d\tilde{\gamma})
\end{eqnarray*}
by Corollary~\ref{c2.4}.
\end{proof}
\begin{sat}\label{Th4.4}
Let $\mu$ be a measure on $ \mathbf{B}(\varGamma)$ and $\alpha \in (0,\infty)$. 
Assume that $\varGamma$ is a Polish space. The set $\mathbf{P}(H)$ is dense in   $L^\alpha(\mu)$ if and only if the measure $\mu$
is concentrated on a transversal.
\end{sat}
\begin{proof}
The ``if-part'' is a special case of Theorem~\ref{Th3.8}. Assume, conversely, that 
 $\mathbf{P}(H)$ is dense in $L^\alpha(\mu)$ for some $\alpha \in (0,\infty)$.  It is enough to show that there exists a function 
 $\tau: \varGamma/\varLambda \to \varGamma$ satisfying the conditions of Lemma~\ref{l4.3}. Choose a transition probability $w$ on
 $ (\varGamma/\varLambda) \times \mathbf{B}(\varGamma)$ satisfying conditions (i) and (ii) of Theorem~\ref{Th2.5} and a set 
 $\tilde{B}\in  \mathbf{B}( \varGamma/\varLambda )$ such that $\tilde{B}_0\subseteq\tilde{B}$, 
 $\mu\pi^{-1}(\tilde{B})= 0$ and  $\mathbf{P}(H)$ is dense in $L^\alpha(w(\tilde{\gamma},\cdot))$ for all 
 $\tilde{\gamma}\in (\varGamma/\varLambda)\setminus \tilde{B}$, cf. Corollary~\ref{c2.6}. Since the functions of $\mathbf{P}(H)$
 are constant on each $\varLambda$-coset, there exists a function $\psi: (\varGamma/\varLambda)\setminus \tilde{B}\to \varGamma$
 satisfying $w(\tilde{\gamma},\cdot)= \delta_{\psi(\tilde{\gamma} )}(\cdot)$ for all 
 $\tilde{\gamma}\in (\varGamma/\varLambda)\setminus \tilde{B}$. It is not difficult  to see that $\psi$ is 
  $(\mathbf{B}(\varGamma/\varLambda),\mathbf{B}(\varGamma) )$-measurable. Choose an arbitrary  
  $(\mathbf{B}(\varGamma/\varLambda),\mathbf{B}(\varGamma) )$-measurable cross-section $\chi$ according to Theorem~\ref{Th3.5}.
  The function $\tau$, which is equal to $\psi$ on $(\varGamma/\varLambda)\setminus \tilde{B}$ and equal to $\chi$ on $\tilde{B}$
  has all  desired properties.
\end{proof}
We conclude the present section  by proving  an analogous result if $\varGamma$ is an arbitrary LCA
group and $\varLambda$ is countable.
\begin{sat}
Let $\mu$ be a measure on $ \mathbf{B}(\varGamma)$ and $\alpha \in (0,\infty)$. 
Assume that $\varLambda$ is countable. The set  $\mathbf{P}(H)$ is dense in   $L^\alpha(\mu)$ if and only if the measure $\mu$
is concentrated on a transversal. 
\end{sat}
 \begin{proof}
 Since $\varLambda$ is countable, it is discrete, cf. \cite[Corollary to Theorem~2]{Morris77}, 
 hence, it is metrizable and the ``if-part'' is a special case of Theorem~\ref{Th3.8}. To prove that $\mu$ is concentrated on a transversal 
 if for some $\alpha \in (0,\infty)$,  $\mathbf{P}(H)$ is dense $L^\alpha(\mu)$, choose a transversal $T\in \mathbf{B}(\varGamma)$
 according to Theorem~\ref{Th3.5}. For $\lambda\in\varLambda$, let $\mu_{\lambda}$ be the restriction of $\mu$ 
 to $ \mathbf{B}(\lambda+T)$, $\tilde{\mu}_{\lambda}$ be the translate of $\mu_{\lambda}$ to $ \mathbf{B}(T)$, i.~e.
 $\tilde{\mu}_{\lambda}(B)= \mu_{\lambda}(\lambda+B)$, $B\in  \mathbf{B}(T)$. Setting
 $\tilde{\mu}_{\lambda}(B):=\tilde{\mu}_{\lambda}(B\cap T)$,  $B\in  \mathbf{B}(\varGamma)$, we can extend  $\tilde{\mu}_{\lambda}$ 
 to $\mathbf{B}(\varGamma)$ and define $ \tilde{\mu}:=\sum_{\lambda\in \varLambda}\tilde{\mu}_{\lambda}$. By Lemma~\ref{l2.1} all 
 measures just defined are regular. Since $ \tilde{\mu}$ is concentrated on a transversal, the linear space  $\mathbf{P}(H)$ is dense in $L^{\alpha}(\tilde{\mu})$. For $\tilde{f}\in L^{\alpha}(\tilde{\mu})$, define 
 $(V \tilde{f})(\gamma):=\tilde{f}(\gamma-\lambda)$ if 
 $\gamma\in \lambda+T$, $\lambda\in \varLambda$. It is not difficult to show that $V$ establishes an isometric isomorphism from 
 $L^\alpha(\tilde{\mu})$ into
 $L^\alpha({\mu})$, that $V$ is the identity on $\mathbf{P}(H)$ and that  $V^{-1}f=f$ $\tilde{\mu}$-a.~e. for all elements $f\in L^\alpha({\mu})$, which belong to the range of $V$, 
 cf. \cite[Lemma~2.2]{Klotz05}. Let $h_{\lambda}$ be a  $(\mathbf{B}(\varGamma),\mathbf{B}([0,\infty) )$-measurable
 Radon-Nikodym derivative of $\tilde{\mu}_{\lambda}$ with respect to $\tilde{\mu}$. 
 We can assume that $h_{\lambda}=0$ on $\varGamma\setminus T$ and define 
 $T_{\lambda}:=\{ \gamma\in \varGamma :h_{\lambda}(\gamma)\neq 0\}$, 
 $T_{\lambda\kappa}:=T_{\lambda}\cap T_{\kappa}$, 
 $\lambda,\kappa\in \varLambda$, $\lambda\neq \kappa$. Since $\mathbf{P}(H)$ is assumed to be dense in $L^\alpha(\mu)$, there exists
 $\tilde{f}\in L^{\alpha}(\tilde{\mu})$ such that $V\tilde{f}=1_{\lambda+T_{\lambda\kappa}}$. From the definition of 
 $V\tilde{f}$ it follows that for  $\tilde{\mu}$-a.~a. $\gamma\in T_{\lambda\kappa}$ one has 
 $\tilde{f}(\gamma)=1_{\lambda+T_{\lambda\kappa}}(\lambda+\gamma)=1$ as well as 
 $\tilde{f}(\gamma)=1_{\kappa+T_{\lambda\kappa}}(\kappa+\gamma)=0$, which yields $\tilde{\mu}(T_{\lambda\kappa})=0$. Set
 $T_{\lambda}'=\bigcup_{\kappa\in \varLambda\setminus \{\lambda\}} T_{\lambda\kappa}$, $T_{\lambda}^{d}:=T_{\lambda}\setminus T_{\lambda}'
 $, $\lambda\in \varLambda$, $D:= \bigcup_{\lambda\in \varLambda}(\lambda+T_{\lambda}^{d})$. Then 
 \begin{eqnarray*}
 \mu(\varGamma\setminus D)&=&\mu\left(\bigcup_{\lambda\in \varLambda}(\lambda+T)\setminus
 \bigcup_{\lambda\in \varLambda}(\lambda+T_{\lambda}^{d})\right)
 \\&=& \mu\left(\bigcup_{\lambda\in \varLambda}((\lambda+T)\setminus(\lambda+T_{\lambda}^{d})\right)
 = \sum_{\lambda\in \varLambda}\mu\left((\lambda+T)\setminus(\lambda+T_{\lambda}^{d})\right)
 \\&=& \sum_{\lambda\in \varLambda}\mu\left(\lambda+(T\setminus T_{\lambda}^{d})\right)= \sum_{\lambda\in \varLambda}\tilde{\mu}_{\lambda}(T\setminus T_{\lambda}^{d})
 = \sum_{\lambda\in \varLambda}\tilde{\mu}_{\lambda}\left((T\setminus (T_{\lambda}\setminus T_{\lambda}')\right))
 \\&=& \sum_{\lambda\in \varLambda}\tilde{\mu}_{\lambda}\left((T\setminus T_{\lambda})\cup T_{\lambda}'\right)
 = \sum_{\lambda\in \varLambda}\tilde{\mu}_{\lambda}\left(T\setminus T_{\lambda}\right)
 + \sum_{\lambda\in \varLambda}\tilde{\mu}_{\lambda}\left(T_{\lambda}'\right)=0.
 \end{eqnarray*}
If $D\cap(\lambda+D)\neq \emptyset$ for some $\lambda\in \varLambda$, there exist $\lambda_{j}\in \varLambda, \: 
\gamma_{j}\in T_{\lambda_{j}}\setminus T_{\lambda_{j}}'$, $j\in \{1,2\}$, such that
$\lambda_{1}+\gamma_{1}=\lambda +\lambda_{2}+\gamma_{2}$, hence, $ \lambda_{1}+(T\cap(\lambda+\lambda_{2}-\lambda_{1}+T))\neq \emptyset $.
Since $T$ is a transversal, we get $\lambda_{1}=\lambda+\lambda_{2}$ and then $\gamma_{1}=\gamma_{2}$, which implies that $\lambda=0$.
Thus, $\mu$ is concentrated on a transversal.
 \end{proof}

\noindent
J. M. Medina\\
Inst. Argentino  de Matem\'atica ``A. P. Calder\'on''- CONICET and Universidad de Buenos Aires-Departamento de Matem\'{a}tica.\\
Saavedra 15, 3er piso (1083)\\
Buenos Aires, Argentina. Contact: jmedina@fi.uba.ar\\
~\\
L. Klotz and M. Riedel\\ 
Fakultät für Mathematik und Informatik\\
Universität Leipzig\\
04109 Leipzig, Germany. Contact: Lutz-Peter.Klotz@math.uni-leipzig.de
\end{document}